\documentclass[12pt]{amsart}
\usepackage{amsmath,amsthm,amsfonts,amssymb}
\begin{document} 
\newcommand{\A}{{\mathbb A}}
\newcommand{\B}{{\mathbb B}}
\newcommand{\C}{{\mathbb C}}
\newcommand{\N}{{\mathbb N}}
\newcommand{\Q}{{\mathbb Q}}
\newcommand{\Z}{{\mathbb Z}}
\renewcommand{\P}{{\mathbb P}}
\renewcommand{\O}{{\mathcal O}}
\newcommand{\R}{{\mathbb R}}
\newcommand{\rc}{\subset}
\newcommand{\rank}{\mathop{rank}}
\newcommand{\trace}{\mathop{tr}}
\newcommand{\dimc}{\mathop{dim}_{\C}}
\newcommand{\Lie}{\mathop{Lie}}
\newcommand{\Spec}{\mathop{Spec}}
\newcommand{\Auto}{\mathop{{\rm Aut}_{\mathcal O}}}
\newcommand{\alg}[1]{{\mathbf #1}}
\newtheorem{lemma}{Lemma}[section]
\newtheorem{definition}[lemma]{Definition}
\newtheorem{corollary}[lemma]{Corollary}
\newtheorem{example}{Example}
\newtheorem*{remark}{Remark}
\newtheorem*{observation}{Observation}
\newtheorem*{remarks}{Remarks}
\newtheorem{proposition}[lemma]{Proposition}
\newtheorem{theorem}[lemma]{Theorem}
\numberwithin{equation}{section}
\def\labelenumi{\rm(\roman{enumi})}
\title{%
Tame discrete subsets in Stein manifolds
}
\author {J\"org Winkelmann}
\begin{abstract}
we generalize the notion of ``tameness'' for discrete sets
to arbitrary Stein manifolds.
\end{abstract}
\address{%
J\"org Winkelmann \\
Lehrstuhl Analysis II \\
Mathematisches Institut \\
NA 4/73\\
Ruhr-Universit\"at Bochum\\
44780 Bochum \\
Germany
}
\email{jwinkel@member.ams.org
}
\maketitle

\section{Introduction}

For discrete subsets in $\C^n$ the notion of being ``tame'' was defined
in the important paper of Rosay and Rudin \cite{RR}. 
A discrete subet $D\subset\C^n$ is called ``tame''
if and only there exists an automorphism $\phi$ of $\C^n$ such that
$\phi(D)=\{0\}^{n-1}\times\Z$.

We want to introduce and study a similar notion for complex manifolds
other than $\C^n$.

Therefore we propose a new definition, show that it is equivalent to that
of Rosay and Rudin if the ambient manifold is $\C^n$ and deduce some
standard properties.

To obtain good results, we need some knowledge on the automorphism group
of the respective complex manifold. For this reason 
we get our best results in the case where 
the manifold is biholomorphic to a complex Lie group. We concentrate
on complex semisimple Lie groups, since every simply-connected complex
Lie group is biholomorphic to a direct product of $\C^n$ and a semisimple
complex Lie group.

\begin{definition}\label{def-tame}
Let $X$ be a complex manifold. 
An infinite discrete subset $D$ is called (weakly) {\em tame} if for every
exhaustion function $\rho:X\to\R^+$ and every map $\zeta:D\to\R^+$
there exists an automorphism $\phi$ of $X$ such that
$\rho(\phi(x))\ge \zeta(x)$ for all $x\in D$.
\end{definition}

Andrist and Ugolini 
(\cite{AU}) have proposed a different notion, namely the
following:

\begin{definition}
Let $X$ be a complex manifold. 
An infinite discrete subset $D$ is called (strongly) {\em tame} 
if for every injective map $f:D\to D$ there exists
an automorphism $\phi$ of $X$ such that $\phi(x)=f(x)$ 
for all $x\in D$.
\end{definition}

It is easily verified that ``strongly tame'' implies ``weakly tame''. For $X\simeq\C^n$
and $X\simeq SL_n(\C)$ both tameness notions coincide. Furthermore,
for $X=\C^n$ both notions agree with tameness as defined by Rosay and 
Rudin.

However, for arbitrary manifolds ``strongly tame'' and ``weakly tame''
are not equivalent.

In this article, unless explicitly stated otherwise, tame always
means weakly tame, i.e., tame in the sense of definition~\ref{def-tame}.

\section{Comparison between $\C^n$ and semisimple Lie groups}

For tame discrete sets in $\C^n$ in the sense of Rosay and Rudin,
the following facts are well-known:
\begin{enumerate}
\item
Any two tame sets are equivalent.
\item
Every discrete subgroup of $(\C^n,+)$ is tame as a discrete set.
\item
Every discrete subset of $\C^n$ is the union of two tame ones.
\item
There exist non-tame subsets in $\C^n$.
\item
Every injective self-map of a tame discrete subsetz of $\C^n$
extends to a biholomorphic self-map of $\C^n$.
\item
If $v_k$ is a sequence in $\C^n$ with
$\sum_{k=1}^\infty \frac{1}{||v_k||^{2n-1}}<\infty$,
then $\{v_k:k\in\N\}$ is a tame discrete subset.
\end{enumerate}

(See \cite{RR} for $(i),(iii),(iv),(v)$, \cite{JW} for $(vi)$.
For $n=2$, $(ii)$ is implied by proposition~4.1. of \cite{BL}. The
proof given there generalizes easily to arbitrary dimension $n$.)

For discrete subsets in semisimple complex Lie groups we are
able to prove the following properties:
\begin{enumerate}
\item
Any two tame discrete subsets in $SL_n(\C)$ are equivalent.
(proposition~\ref{sln-equiv})
\item
Certain discrete subgroups may be verified to be tame discrete
subsets. In particular,  $SL_2(\Z[i])$ is a tame discrete
subset (corollary~\ref{sl2zi}), and also every discrete subgroup of a one-dimensional
Lie subgroup of $SL_n(\C)$ (proposition~\ref{one-para}) and every
discrete subgroup fo a maximal torus (corollary~\ref{torus-tame}).
\item
Every discrete subset of $SL_n(\C)$ is the union of $n$
tame discrete subsets.
(corollary~\ref{union-tame})
\item
Every semisimple complex Lie group admits a non-tame discrete subset.
(proposition~\ref{ex-non-tame})
\item
Every injective self-map of a tame discrete subset of $SL_n(\C)$
extends to a biholomorphic self map of $SL_n(\C)$.
(proposition~\ref{sln-equiv}.)
\item
For every semisimple complex Lie group $S$ there exists a 
``threshold sequence'', i.e., there exists a sequence
of numbers $R_k>0$ and an exhaustion function $\tau$ such that
every sequence $g_k$ with $\tau(g_k)>R_k$ defines a tame discrete
subset.
(proposition~\ref{gen-proj}).
\end{enumerate}

\section{Results for other manifolds}
While tame discrete sets in semisimple complex Lie groups behave
in a way very similar to those in $\C^n$, for arbitrary complex
manifolds the situation is quite different:

\begin{enumerate}
\item
On $\C^n\setminus\{(0,\ldots,0)\}$ ($n\ge 2$)
there do exist discrete subsets
which may not be realized as a finite union of tame discrete subsets.
(corollary~\ref{no-finite-union}.)
\item
On $\Delta\times\C$ there are (weakly) tame discrete sets which are not
strongly tame. There are permutations of tame discrete sets which
do not extend to biholomorphic self-maps of the ambient manifold.
(proposition~\ref{no-strong},\ref{class-d-c})
\item
On $\Delta\times\C$ there exist inequivalent tame discrete subsets.
(corollary~\ref{inequiv-tame}.)
\item
On $\C^n\setminus\{(0,\ldots,0)\}$ there is no ``threshold sequence''.
(corollary~\ref{no-thres}.)
\end{enumerate}

\section{Preparations}

\begin{proposition}
Let $X$ be a complex manifold and let $D$ be an infinite discrete subset.

Then $D$ is tame if and only if there exists {\em one} exhaustion function
$\rho$ 
such that the following property holds:
{\em For every map $\zeta:D\to\R^+$
there exists an automorphism $\phi$ of $X$ such that
$\rho(\phi(x))\ge \zeta(x)$ for all $x\in D$.}
\end{proposition}
(In the definition~\ref{def-tame} for being tame it is required that this
property holds for {\em every} exhaustion function.)
\begin{proof}
Assume that the property holds with respect to 
a given exhaustion function $\rho$. We have to show that
$D$ is tame, i.e., that the property holds with respect to every
exhaustion function.
Let $\tau$ be an arbitrary exhaustion function and let $\zeta:D\to\R$
be a map. We choose a map $\zeta_0:D\to\R$ in such a way
that
\[
\zeta_0(x) > \sup \{ \rho(p):\tau(p)< \zeta(x)\} \ \ \forall x\in D
\]
By assumption there is an automorphism $\phi$ of $X$ such that
$\rho(\phi(x))\ge \zeta_0(x)$ for all $x\in D$.
By the construction of $\zeta_0$ it follows that
\[
\phi(x) \not\in \{p\in X: \tau(p) <\zeta(x) \} \ \ \forall x\in D.
\]
Hence $\tau(\phi(x)) \ge \zeta(x)$ for all $x\in D$ as desired.
\end{proof}

\begin {definition}
Let $X$ be a complex manifold.
Two sequences $A(k)$, $B(k)$ in $X$ are called {\em equivalent}
if there exists a holomorphic automorphism $\phi$ of the complex
manifold $X$ such that $\phi(A(k))=B(k)$ for all $k\in\N$.

A sequence $A(k)$ in $X$ is called {\em tame} if the set $\{A(k):k\in\N$
is a tame discrete subset of $X$.
\end {definition}

\begin{proposition}
Let $X$ be a complex manifold. Assume that the automorphism group
$\Auto(X)$ is a finite-dimensional Lie group with countably
many connected components.

Then $X$ does not admit any tame discrete subset.
\end{proposition}

\begin{proof}
Let  $D=\{p_n:n\in\N\}$ be an infinite discrete subset of $X$.
Fix an exhaustion function $\rho:X\to\R^+$.
Let $K_n$ be an increasing sequence of compact subsets of 
$G=\Auto(X)$, which exhausts $G$, i.e., $\cup_nK_n=G$.
Define $c_n=\max\{\rho(x):x\in K_n(p_n)\}$. Choose $\zeta_n>c_n$.
By construction, if $\rho(\phi(p_n))>\zeta_n$ for some $n\in\N$, 
then $\phi\in G\setminus K_n$.
Since $\cap_n(G\setminus K_n)=\{\}$, it follows that there is no
$\phi\in\Auto(X)$ with $\rho(\phi(p_n))>\zeta_n$ for all $n\in\N$.
Thus $D$ is not tame.
\end{proof}

\begin{corollary}\label{hyper}
Complex manifolds which are hyperbolic in the sense of Kobayashi
(e.g.~bounded domains in Stein manifolds) do not admit tame
subsets.
\end{corollary}

\begin{corollary}
Let $\bar X$ be a compact complex manifold, $\dimc(X)\ge 2$,
 and let $S$ be a finite
subset.

Then $X=\bar X\setminus S$ contains no tame discrete subset.
\end{corollary}

\begin{proof}
Every automorphism of $X$ extends to an automorphism of $\bar X$ and the
automorphism group of $\bar X$ is a finite-dimensional Lie group
by the theorem of Bochner and Montgomery.
\end{proof}

\begin{corollary}
There are no tame discrete subsets in Riemann surfaces.
\end{corollary}

\begin{proposition}
A discrete subset $D$ of  $\C^n$ is tame in the sense of 
definition~\ref{def-tame} if and only if it is tame in the sense of
Rosay and Rudin, i.e., if and only if there exists a holomorphic
automorphism $\phi$ of $\C^n$ such that $\phi(D)=\Z\times\{0\}^{n-1}$.
\end{proposition}

\begin{proof}
For every map $\xi:\Z\to\R$ there exists a holomorphic function
$f$ on $\C$ with $f(n)=\xi(n)$ for all $n\in\Z$. The automorphism
$z\mapsto (z_1,z_2+f(z_1),z_3,\ldots,z_n)$ maps $\Z\times\{0\}^{n-1}$
to $\{(n,\xi(n),0,\ldots,0):n\in\Z\}$. 
Using this fact it is clear that
$\Z\times\{0\}^{n-1}$ is tame in the sense of definition~\ref{def-tame}.
Therefore being tame in the sense of Rosay and Rudin implies being
tame in the sense of definition~\ref{def-tame}.

Conversely, if 
a discrete set $D=\{a_k:k\in\N\}$ is 
tame in the sense of definition~\ref{def-tame},
there exists a biholomorphic map $\phi$
of $\C^n$ such that $||\phi(a_k)||^{2n-1}>k^2$ for all $k$.
Then the proposition below implies that $D$ is tame in the sense
of Rosay and Rudin, since $\sum_k k^{-2}<\infty$.
\end{proof}

\begin{proposition}
Let $v_k$ be a sequence in $\C^n$.
If 
\[
\sum_{k=1}^\infty \frac{1}{||v_k||^{2n-1}}<\infty,
\]
then $D=\{v_k: k\in\N\}$ is a tame (in the sense of Rosay and Rudin)
discrete subset of $\C^n$.
\end{proposition}
\begin{proof}
See \cite{JW}.
\end{proof}

\begin{definition}\label{def-ts}
Let $X$ be a complex manifold with an exhaustion function $\rho$.
A sequence of positive real numbers $R_n$ is called
``threshold sequence'' for $(X,\rho)$ if every discrete subset $D$ with
\[
\#\{x\in D:\rho(x)\le R_n\}< n\ \forall n\in\N
\]
is tame.
\end{definition}

In other words: If $R_n$ is a threshold sequence, then every
sequence $x_k$ in $X$ with $\rho(x_k)\ge R_k$ defines a
tame discrete subset of $X$.

It follows from \cite{JW} that
every sequence $(R_k)$ with $\sum_k (R_k)^{-(2n-1)}<\infty$
 is a threshold sequence for $\C^n$
(with respect to the exhaustion function $\rho(x)=||x||$.)

If a complex manifold $X$ with exhaustion function $\rho$
admits a threshold sequence $R_k$ and $\tilde\rho$ is
a different exhaustion function, we may define a threshold
sequence $\tilde R_k$ for $(X,\tilde\rho)$ as follows:
We need to ensure that $\rho(x)>R_k$ implies $\tilde\rho(x)
>\tilde R_k$. Hence we may define 
\[
\tilde R_k=\max\{ \tilde\rho(x): \rho(x)\le R_k\}.
\]

Thus, if there exists a threshold sequence for one exhaustion function,
there also exists a threshold sequence for any other exhaustion
function on the same complex manifold, i.e., whether or not there
exists a threshold sequence depends only on the complex manifold,
not on the exhaustion function.

We will see that there exist threshold sequences for every semisimple
complex Lie group (proposition~\ref{gen-proj}).

In contrast, there is no threshold sequence for $\C^n\setminus\{(0,
\ldots,0) \}$ (corollary~\ref{no-thres}).

\begin{proposition}\label{unbounded}
Let $X$ be a complex manifold for which there exists a ``threshold
sequence''. Let $A\subset X$ be an unbounded (i.e.~not relatively
compact) subset.

Then $A$ contains a subset which is a tame discrete subset of $X$.
\end{proposition}

\begin{proof}
Obvious.
\end{proof}

\begin{corollary}
Let $S$ be a complex semisimple Lie group, 
let $\Omega$ be a Stein open subset
with $\Omega\ne S$
and let $D$ be a tame discrete subset of $S$. Then there exists
a holomorphic automorphism $\phi$ of $S$ such that $\phi(\Omega)\cap
D=\{\}$.
\end{corollary}

\begin{proof}
The manifold $S$ admits a threshold sequence (cf.~prop.~\ref{gen-proj}).
Furthermore $S\setminus\Omega$ is unbounded due to 
``Hartogs Kugelsatz'', because $\Omega$
is assumed to be Stein.
\end{proof}

\section{The case $X=\Delta\times\C$}
We start by deducing a description of the automorphisms.
Let $\Delta=\{z\in\C:|z|<1\}$. On $X=\Delta\times\C$ there is a natural
equivalence relation: Two points $(x,y)$, $(z,w)$ can be separated by
a bounded holomorphic function if and only if $x\ne z$. The projection
onto the first factor is therefore equivariant for every automorphism of 
$X$. Using this fact one easily verifies that every automorphism of $X$
can be written in the form
\[
(z,w)\mapsto (\phi(z),f(z)w+g(z))
\]
with $\phi\in\Auto(\Delta)$, $f\in\O^*(\Delta)$ and $g\in\O(\Delta)$.

\begin{proposition}\label{class-d-c}
Let  $X=\Delta\times\C$  with $\Delta=\{z\in\C:|z|<1\}$.

A discrete subset $D\subset X$ is tame if and only if 
\[
\pi_1(D)=\{z\in\Delta: \exists (z,w)\in D\}
\]
is discrete in $\Delta$ and
\[
\pi^{-1}(p)=\{(z,w)\in D:z=p\}
\]
is finite for all $p\in \Delta$.
\end{proposition}

\begin{proof}
The discrete sets fulfilling these conditions are tame due to 
proposition~\ref{pi-tame-is-tame}.

Suppose conversely that $D$ is a discrete subset which does not
fulfill all of these conditions.
Then
there exists a divergent sequence $(p_n,q_n)$ in $D$ 
with $\lim p_n=p\in\Delta$.
We fix an exhaustion function $\tau$ on $X$ such that
$\tau(z,w)=|w|+(1-|z|)^{-1}$.
We choose $R_n$ such that $R_n>2^n|q_n|$.
By the tameness assumption there exists an automorphism 
of $X$ given as
\[
(z,w)\mapsto (\phi(z), f(z)w+g(z)).
\]
such that 
\[
\tau(\phi(p_n),f(p_n)q_n+g(p_n))>R_n \ \forall n.
\]
Since $\lim p_n=p$, we have 
\[
\lim (1-\phi(p_n))^{-1}=(1-\phi(p))^{-1}
\]
Thus there exists a constant $K$ such that 
$(1-\phi(p_n)^{-1}<K$ for all sufficiently large $n$.
Then
\[
|f(p_n)q_n+g(p_n)| >R_n-K>2^n|q_n|-K
\]
With $\lim p_n=p$ we obtain
\[
\lim\frac{1}{|q_n|}|f(p_n)q_n+g(p_n)|=|f(p)|
\]
which contradicts
\[
\frac{1}{|q_n|}|f(p_n)q_n+g(p_n)| >\frac{1}{|q_n|}\left(R_n-K>2^n|q_n|-K\right)
=2^n-\frac{K}{|q_n|},
\]
since $\lim\left(2^n-\frac{K}{|q_n|}\right)=+\infty$.
\end{proof}

\begin{corollary}\label{inequiv-tame}
The manifold $X=\Delta\times\C$ admits inequivalent tame discrete subsets.
\end{corollary}

\begin{proof}
Every automorphism of the unit disc preserves the Poincar\'e metric.
Hence it is clear that there are many inequivalent discrete subsets
in the unit disc. By the proposition, for each discrete subset
$D\subset\Delta$ we obtain a tame discrete subset $D'$ in $X$ 
via $D'=D\times0\}$.
\end{proof}

Thus $X=\Delta\times\C$ admits many (weakly) tame discrete subsets.
In contrast, there are no strongly tame discrete subsets:

\begin{proposition}\label{no-strong}
There are no strongly tame discrete subsets in $X=\Delta\times\C$.
\end{proposition}

\begin{proof}
Suppose $D$ is strongly tame. Then it is tame. Due to 
proposition~\ref{class-d-c}
$\pi_1(D)\subset\Delta$ is infinite. Thus for any two points $v,w\in D$
we can find an injective map $F:D\to D$ with $\pi_1(F(v))\ne\pi_1(F(w))$.
It follows that $\pi_1|_D$ is injective. However, now any injective self-map
of $\pi_1(D)$ is induced by an injective self-map of $D$. Thus the assumption
of $D$ beeing strongly tame in $X$ implies that $\pi_1(D)$ is strongly
tame (and therefore tame) in $\Delta$, which is impossible 
(cf.~corollary \ref{hyper}).
\end{proof}

\section{The case $\C^n\setminus\{(0,\ldots,0\}$}

We start with a preparation.

\begin{lemma}
Let $\phi$ be a $C^1$-diffeomorphism of $\C^n$ fixing the origin.

Then there exists an open neighborhood $W$ of the origin $(0,\ldots,0)$
and constants $C_1,C_2>0$
such that
\[
C_1||v||\le ||\phi(v)|| \le C_2||v||\ \ \forall v\in W
\]
\end{lemma}

\begin{proof}
Let $U$ be a convex relatively compact open neighborhood of $(0,\ldots,0)$.
Define $C=\sup_{v\in U}\max\{||D\phi||_v, ||D(\phi^{-1})||_v\}$.
Then $||\phi(v)|| \le C||v||$ and $||\phi^{-1}(v)||\le C||v||$ for all
$v\in U$. This implies the assertion with $C_2=C$, $C_1=1/C$ and
$W=U\cap\phi^{-1}(U)$.
\end{proof}

\begin{proposition}
A discrete subset $D$ in $X=\C^n\setminus\{(0,\ldots,0\}$ is {\em tame}
if and only if it is discrete and tame considered as a subset of $\C^n$.
\end{proposition}

\begin{proof}
We recall that holomorphic automorphisms of $X$ extend to holomorphic
automorphisms of $\C^n$.

We fix the exhaustion function $\tau:X\to\R^+$ given by
$\tau(x)=\max\{||x||,||x||^{-1}\}$.

Assume that $D$ is tame and discrete in $\C^n$.
Then $D_1=D\cup\{0\}$ is likewise tame. 
Enumerate the elements of $D$ such that
$D= \{ a_k:k\in\N\}$.
Let a 
strictly increasing sequence $R_k$ be given.
Choose elements $v_k\in\C^n$ such that $R_{k+1}\ge||v_k||>R_k$ for all 
$k\in\N$. Since $D_1$ is tame, there exists an automorphism
$\phi$ of $\C^n$ such that $\phi(0)=0$ and $\phi(a_k)=v_k$ for all
$k\in\N$. This shows: For any such sequence $R_k$ there exists an
automorphism of $X$ such that $||\phi(a_k)||>R_k$ for all $k$, i.e.,
$D$ is tame in $X$.

Now assume that $D$ is a discrete subset in $X$ which is not discrete
in $\C^n$. Then there exists a sequence $\gamma_k\in D$ with
$\lim\gamma_k=(0,\ldots,0)$. We want to show that $D$ is not tame.
We choose a map $\zeta:D\to\R$ such that $\zeta(\gamma_k)>k/||\gamma_k||$
for all $k$. We claim that there exists no holomorphic automorphism 
$\phi$ of $X$ 
such that $\tau(\phi(\gamma_k))\ge \zeta(\gamma_k)$ for all $k\in\N$.
Indeed each holomorphic automorphisms of $X$ extends to an automorphism
of $\C^n$ fixing the origin $(0,\ldots,0)$.
It follows that there exists a neighbourhood $W$ of $(0,\ldots,0)$ in
$\C^n$ and constants $C_2>C_1>0$ such that
\[
C_2 ||v|| \ge ||\phi(v)|| \ge C_1||v||\ \ \forall v\in W.
\]
Assume that $W$ is contained in the unit ball. Then $\tau(v)=1/||v||$
for all $v\in W$ and therefore
\[
\frac{1}{C_2||v||} \le \tau(\phi(v)) \le \frac{1}{C_1||v||}
\ \ \forall v\in W\setminus\{(0,\ldots,0)\}
\]
This implies that $\tau(\gamma_k)<\zeta(\gamma_k)$ for all
$k$ with $\gamma_k\in W$ and $k>1/C_1$.
Hence $D$ can not be tame.
\end{proof}

\begin{corollary}\label{no-thres}
There is no threshold sequence for $X=\C^n\setminus\{(0,\ldots,0)\}$.
\end{corollary}

\begin{proof}
Let $v_n$ by any sequence in $\C^n\setminus\{(0,\ldots,0)\}$ which
converges to the origin in $\C^n$. If there would exist a treshold
sequence, $(v_n)$ would contain a tame subsequence, in contradiction
to the above proposition.
\end{proof}

\begin{corollary}\label{no-finite-union}
There exists discrete subsets of $X$ which can not be realized
as the union of finitely many tame discrete subsets.
\end{corollary}

\begin{proof}
Just take any discrete subset of $X$
 which in $\C^n$ has an accumulation point
in $(0,\ldots,0)$.
\end{proof}

\section{Some preparation}
\begin{lemma}
Let $G$ be a connected complex Lie group. Then there exists
a surjective holomorphic map from some complex vector space $\C^n$
onto $G$.
\end{lemma}
\begin{corollary}\label{prescribe-on-discrete}
Let $G$ be a connected complex Lie group, $X$ a Stein complex manifold
and $D\subset X$ a discrete subset. Then every map $f:D\to G$ extends
to a holomorphic map $F:X\to G$.
\end{corollary}

\begin{proof}
We fix a surjective holomorphic map $\Phi:\C^n\to G$. Then given
a map $f:D\to G$, 
there is a ``lift'' $g:D\to \C^n$, i.e., a map $g:D\to \C^n$
such that $f=\Phi\circ g$.
The existence of the desired map $F$ now follows from the classical
fact in complex analysis that the value of a holomorphic function
can be described on a fixed discrete set.
\end{proof}

\section{$\pi$-tame sets}

\begin{definition}
Let $X,Y$ be complex manifolds and $\pi:X\to Y$ a holomorphic map.
An infinite discrete subset $D\subset X$ is called {\em $\pi$-tame} 
if there exists an automorphism $\phi$ of $X$ such that
the restriction of the map $\pi\circ\phi$ to $D$ is proper.
\end{definition}

\begin{remark}
A map from a discrete space $D$ into a locally compact space $Y$ is
{\em proper} if and only if it has discrete image and finite fibers.
\end{remark}

\begin{remark} 
If $D$ is $\pi$-tame, and $D'$ is discrete infinite with
$D'\setminus D$ being finite, then $D'$ is $\pi$-tame.
\end{remark}

\begin{proposition}\label{pi-tame-is-tame}
Let $H$ be a non-compact connected complex Lie group and let $\pi:X\to Y$
be a $H$-principal bundle. If $Y$ is Stein, then every $\pi$-tame discrete
subset $D\subset X$ is tame.
\end{proposition}

\begin{proof}
Let $\rho$ be an exhaustion function on $X$.
Assume that $\pi|_D$ is proper. Let $q\in \pi(D)$ and $X_q=\pi^{-1}(q)$.
Observe that $X_q\cap D$ is finite and that $\rho|_{X_q}$ is unbounded.
Hence there exists an element $h_q\in H$ (acting on $X_q$ by the principal
$H$-action of the principal bundle) with
\[
\rho(ph_q)\ge\zeta(p)\ \forall p\in X_q\cap D.
\]
Next we choose an holomorphic map $F:Y\to H$ with $F(q)=h_q$
for all $q\in\pi(D)$. 
(Such a map $F$ exists due to corollary~\ref{prescribe-on-discrete}.)
Now we can define the desired automorphism $\phi$ as the principal action
of $F(q)$ on $X_q$ for each $q\in Y$.
(In other words: $\phi:x\mapsto x\cdot F(\pi(x))$.)
\end{proof}

\begin{proposition}
Let $X\to Y$ be a principal bundle with non-compact structure group $G$.
Assume that there exists a non-constant holomorphic function on $Y$.

Then there exists a tame discrete subset in $X$.
\end{proposition}

\begin{proof}
Let $f$ be a non-constant holomorphic function on $Y$ and $W=f(Y)$.
Note that $W$ is an open domain in $\C$ and in particular Stein.
Choose a sequence $x_n$ in $X$ such that $f(\pi(x_n))$ is without
accumulation points 
in $W$. Assume that $\rho$ is an exhaustion function on $X$ and let
$\zeta:D\to\R^+$ be a given map where $D=\{x_n:n\in\N\}$.
Since $G$ is non-compact and acting with closed orbits, we can find
elements $g_n\in G$ such that $\rho(x_ng_n)>\zeta(x_n)$.
Recall that for every connected complex Lie group $G$ there exists
a surjective holomorphic map $h:\C^N\to G$. Using this and the fact
that $f(\pi(D))$ is discrete in $W$, it is clear that there
exists a holomorphic map $F:W\to G$ with $F(f(\tau(x_n)))=g_n$ for all
$n\in\N$. Now $\phi:x\mapsto x\cdot F(f(\tau(x)))$ defines an
automorphism of $X$ such that $\rho(\phi(x))>\zeta(x)$ for all
$x\in D$. Since $\zeta:D\to\R^+$ was arbitrary, we may conclude
that $D$ is tame.
\end{proof}

\section{Generic projections}

\subsection{Generalities}
\begin{proposition}\label{prop-LA}
Let $V$, $W$ be finite-dimensional (real) Hilbert spaces, $V\ne\{0\}$, 
and  let $\pi:V\to W$ be a linear map.
Let $K$ be a connected compact Lie group which acts linearly
and orthogonally on $V$.
We assume that $\ker \pi$ contains no non-trivial $K$-invariant
vector subspace. 

Let $\mu$ denote the Haar measure on $K$
(normalized, i.e., $\mu(K)=1$).

Then for every $r,\delta>0$ there exists a number $R>0$ such that
\[
\mu\{k\in K: ||\pi(k(v))||<r\}<\delta
\]
for all $v\in V$ with $||v||>R$.
\end{proposition}
\begin{proof}
There is no loss in generality in assuming
$||\pi(v)||\le||v||$ for all
$v\in V$.

Let $S=\{v\in V:||v||=1\}$. Note that $S$ is compact.
We define an auxiliary function 
$f:K\times S\to\R$ as
\[
f(k,v)=||\pi(k\cdot v)||.
\]
Since $K$ does not stabilize any non-trivial vector subspace of
$\ker\pi$, it is clear that 
\[
\forall v\in S:\exists k\in K: f(k,v)>0.
\]
Thus $\{(k,v):f(k,v)=0\}$ is a real-analytic subset of $K\times S$ not 
containing any of the fibers of the projection $pr_2:K\times S\to S$.
Consequently
\begin{equation}\label{nullset}
\mu\left(\{k\in K:f(k,v)=0\}\right)=0
\end{equation}
for every $v\in S$. (Here we use the fact that nowhere dense real analytic
subsets of $K$ are of Haar measure zero.)

Next we define
\[
\Omega(v,r)=\{k\in K:f(k,v)<r\}
\]
and
\[
h(v,r)=\mu\left(\Omega(v,r)\right)
\]
for $(v,r)\in S\times]0,1]$.
Evidently $r\mapsto h(v,r)$ is a monotonously
increasing function for every fixed $v$.
 
On the other hand  $||\pi(v)||\le||v||\ \forall v\in V$ implies
\[
||\pi(kv)-\pi(kw)||\le||v-w||\ \forall k\in K,v,w\in V
\]

Thus $|f(k,v)-f(k,w)|\le||v-w||$ for all $k\in K$, $v,w\in S$.
This in turn implies
\[
\Omega(v,r)\subset \Omega(w,r+\epsilon)
\]
for $v,w\in S$ with $||v-w||<\epsilon$. 
From this it follows that
\begin{equation}\label{eqh}
h(v,r)\le h(w,r+||v-w||)\ \forall v,w,r.
\end{equation}

Next we define another auxiliary function
$g:]0,1]\to[0,1]$:
\[
r\mapsto g(r)=\sup_{v\in S}h(v,r).
\]

We claim that $\lim_{r\to 0} g(r)=0$.

First we note that the limit $\lim_{r\to 0}g(r)$ exists, because $g$
is mononously increasing and bounded from below: $g\ge 0$.

Define $c=\lim_{r\to 0}g(r)$. Then there exists a sequence $r_n$ in $]0,1]$
with $\lim_{n\to\infty}r_n=0$ and $g(r_n)\ge c\ \forall n$.

Now assume $c>0$. Then the definition of $g$ implies the existence
of a sequence $v_n$ in $S$ with $h(v_n,r_n)>c/2$.
Since $S$ is compact, we may assume that $v_n$ is convergent:
$\lim_{n\to\infty} v_n=v$.

Due to inequality~\ref{eqh} we may deduce
\[
h(v,r_n+||v_n-v||)\ge h(v_n,r_n)>c/2.
\]
Since $\lim_{n\to\infty} (r_n+||v_n-v||)=0$, $\sigma$-additivity of the Haar measure
now
implies
\[
\mu\{k\in K: f(kv)=0\}=\mu\left(\cap_n\{k:f(kv)<r_n+||v_n-v||\}\right)\ge c/2
>0
\]
contradicting equation~\ref{nullset}.

Therefore the claim must hold, i.e., $\lim_{r\to 0}g(r)=0$.

Now we can prove the statement of the proposition:
Given $r,\delta>0$, we choose $\rho\in]0,1]$ such that
$g(\rho)<\delta$. Then we define $R=r/\rho$.

Let $v\in V$ with $||v||>R$ and define $v_0=\frac{v}{||v||}\in S$.
We observe that
$||\pi(k(v))||<r$ is equivalent to $||\pi(k(v_0))||<\frac{r}{||v||}$.

Now $g(\rho)<\delta$ implies that $h(w,\rho)<\delta$ for all $w\in S$.
Thus $h(v_0,\rho)<\delta$, i.e., 
\[
\mu\{k\in K: ||\pi(k(v_0))||<\rho\}<\delta
\]
which is equivalent to
\[
\mu\{k\in K: ||\pi(k(v))||<\rho ||v||\}<\delta.
\]
This completes the proof, since $\rho||v||>\rho R=r$.
\end{proof}

\begin{remark}
Instead of a compact real Lie group $K$ with its Haar measure $\mu$
we may take an arbitrary connected real Lie group with an arbitrary
probability measure of Lebesgue measure class.
\end{remark}

\begin{proposition}\label{prop-Rn}
Let $K$ be a Lie group acting on a manifold $S$ and 
let $\pi:S\to Y$ be a continuous map. Let $\rho:S\to\R^+$
be an exhaustion function and let $\tau:Y\to\R^+$ be an
arbitrary continuous function.

Assume that $K$ is endowed with a probability measure $\mu$
such that the following property holds:

{\em ``$(*)$ For every $r,\delta>0$ there exists a number $R>0$ such that
\[
\mu\{k\in K:\tau(\pi(kx))<r\}<\delta
\]
for all $x\in S$ with $\rho(x)>R$.''}

Then there exists a sequence $R_n>0$ such that for every sequence
$x_n$ in $S$ with $\rho(x_n)>R_n\ \forall n$ we can find an element
$k\in K$ such that $\tau(\pi(kx_n))\ge n$ for all $n\in\N$.
\end{proposition}

\begin{proof}
Using $(*)$ we may (for every $n\in\N$) choose $R_n>0$ such that
\[
\mu\{k\in K:\tau(\pi(kx_n))<n\}<2^{-(n+1)}.
\]
forall $x_n$ with $\rho(x_n)>R_n$.

Since $\sum_{n\in\N} 2^{-(n+1)}=\frac 12<1$, the
set
\[
\{k\in K:\tau(\pi(kx))\ge n\forall n\in\N\}
\]
has positive measure (and is therefore non-empty)
for every sequence $x_n$ in $S$ with $\rho(x_n)>R_n$.
\end{proof}

\begin{corollary}
Choose $R_n$ as above and let $D=\{x_n:n\in\N\}$ be discrete
subset of $S$ for which $\rho(x_n)>R_n$ for all $n\in\N$.

Let $\Omega$ denote the set of all $k\in K$ for which $\pi$
restricts to a proper map on $k(D)$, i.e., for which $\pi|_{k(D)}$
has finite fibers and maps to a discrete subset of $Y$.

Then $K\setminus\Omega$ has measure zero.
\end{corollary}

\begin{proof}
By construction
\[
\mu\{k\in K:\tau(\pi(kx_n))<n\}<2^{-(n+1)}.
\]
Let $\Omega_N$ (for a given $N\in\N$) denote the set
\[
\{k\in K:\tau(\pi(kx_n))\ge n\ \forall n\ge N\}
\]
We observe that $\pi$ restricted to $k(D)$ is proper for
any $k$ for which there exists a number $N$ such that $k\in\Omega_N$.
By construction
\[
\mu(K\setminus\Omega_N)\le 2^{-(N+1)}
\]
This implies that
\[
K\setminus\Omega\subset \cap_N\left(K\setminus \Omega_N\right)
\]
is a set of measure zero, 
because $\mu(K\setminus\Omega)\le 2^{-(N+1)}\ \forall N\in\N$.
\end{proof}
\subsection{Generic projections on semisimple complex Lie groups}
\begin{proposition}\label{embed-S}
Let $S$ be an semisimple complex Lie group,
$T\simeq(\C^*)^d\subset S$ and let $K$ be a maximal compact subgroup
of $S$.

Let $\pi$ denote the projection $\pi:S\to S/T$.

Then there are finitely many regular functions $f_i$ on $S/T$
such that
\begin{enumerate}
\item
The map $(f_1,\ldots,f_s):S/T\to\C^s=W$ is a proper embedding.
\item
All functions of the form $x\mapsto f_i(\pi(xk))$ ($k\in K$,
$i\in I=\{1,\ldots, s\}$) generate a finite-dimensional vector
subspace $V$ of $\C[S]$.
\item
The natural map from $S$ to $V^*$ which associates to each point
$p\in S$ the evaluation homomorphism $f\mapsto f(p)$
defines a proper finite morphism from $S$ to $V^*$.
\item
The kernel of the natural map $V^*\to W^*$ contains no 
non-trivial
$K$-invariant vector subspace.
\end{enumerate}
\end{proposition}

\begin{proof}
First we  note that $Y=S/T$ is an affine variety, because both $S$ and
$T$ are reductive. This yields property $(i)$.

Now we define $V$ as in $(ii)$. Finite-dimensionality follows from
standard results on transformation groups.

We obtain a natural map from $S$ to $V^*$ which associates to each point
$p\in S$ its evaluation map $f\mapsto f(p)$.

We have to show $(iii)$.
As a preparation we consider $Z=\cap_{k\in K} k(T)k^{-1}$.
This is an algebraic subgroup of $S$,
obviously normalized by $K$ and therefore normal in $S$.
Since $S$ is semisimple,
$T\simeq(\C^*)^d$ can not be normal in $S$.
Hence $\dim(Z)=0$
and consequently $Z$ is finite.

Let $x,y\in S$. If $f(\pi(xk))=f(\pi(yk))$ for all $f\in W$, $k\in K$,
then necessarily $\pi(xK)=\pi(yk)$ for all $k\in K$.
This is equivalent to $y^{-1}x\in kTk^{-1}$.
It follows that $F(x)=F(y)$ for all $F\in V$ iff $y^{-1}x
\in Z$.
Thus $S\to V^*$ is induced by an injective map from $S/Z$ to $V^*$.

We still have to check properness.

We recall that the base $S/T$ is properly embedded into $W^*$.
Choose a point $\lambda\in W^*$ which is in the image of $S/T$, i.e.,
is the image of some $sT$
and consider its preimage in $V^*$.

Taking note that  all maps between 
$S$, $S/T$, $V^*$ and $W^*$ are $S$-equivariant,
we have to show that $T$-orbits in the fiber in $V^*$ over $\lambda\in W^*$
are closed.

Using theory of toric varieties, we know that if a $T$-orbit in $V^*$ is
not closed, there is a subgroup $\C^*\simeq T_0\subset T$ with a non-closed
$T_0$-orbit.

However, each such $T_0$ is contained in some $SL_2(\C)\simeq H\subset S$.
Hence there exists an element $k\in K$ such that conjugation by $k$
induces the automorphism $z\mapsto 1/z$ of $\C^*\simeq T_0$. It follows
that the closure of a non-trivial $T_0$-orbit in $V^*$ is either the orbit 
itself or isomorphic to $\P_1$. Since $\P_1(\C)$ as a compact Riemann surface
can not be embedded in a complex vector space like $V^*$, it follows
that $T_0$-orbits are closed. Therefore $T$-orbits are closed and
consequently the image of $S$ in $V^*$ is closed.
This verifies $(iii)$.

Finally we have to show $(iv)$. Assume that $C$ is such a 
$K$-invariant vector subspace of the kernel of the map $V^*\to W^*$.

The linear form on $V$ which is defined by an element $c\in C$ must
vanish on $W$, because $C$ is contained in the kernel.
But $K$-invariance of $C$ implies that $c$ must in fact vanish
on the smalles vector subspace of $V$ which contains $W$ and is
$K$-invariant. However, this is $V$ itself, hence $C=\{0\}$.
\end{proof}

\begin{lemma}\label{away-center}
Let $S$ be a complex semisimple Lie group and $Z$ its center.
Let $\tau:S\to S/Z$ denote the natural projection.
Let $D\subset S$ be a tame discrete subset.
Let $R_n$ be a sequence of positive real numbers.

Then  there exists a biholomorphic automorphism $\phi$  of $S$ 
such that 
\begin{enumerate}
\item
$\tau$ is injective on $\phi(D)$.
\item
The condition $\#\{x\in D: \rho(\phi(x))\le R_n\}\le n$
holds for all $n\in\N$.
\end{enumerate}
\end{lemma}

\begin{proof}
First note that $\tau(g)=\tau(h)$ if and only if $g^{-1}h\in Z$.

We choose a (positive-dimensional) unipotent subgroup $U$ of $S$.
Let $\pi:S\to S/U$=Y be the natural projection.
The quotient $S/U$ is a quasi-affine variety, because $U$ is unipotent.
In particular, $Y=S/U$ admits an injective holomorphic map
into some $\C^m$. As a consequence, there is a finite-dimensional
complex vector space $V$ of holomorphic functions on $Y$
separating the points. Let $B$ be the unit ball in $V$ for some
(arbitrary) norm on $V$. Note that $B$ is relatively compact and open.

To each $F\in V$ we may associate an automorphism $\phi_F$ of
$S$ via $g\mapsto g\cdot(F(\pi(g))$.
(The analoguous automorphisms for $\C^n$ are known as ``shears''.)
Observe that $\phi_{-F}\circ\phi_F=id_S$. This confirms that
$\phi_F$ is indeed an automorphism.

Since $\rho$ is an exhaustion function, $K_r=\{x\in S:\rho(x)\le r$
is compact for all $R>0$. As a consequence, 
we may define
\[
\tilde R_n=\max_{F\in\bar B,\rho(x)\le R_n}\rho(\phi_{-F}(x))
\]
Observe that $\rho(p)>\tilde R_n$ implies that $\rho(\phi_F(p))>R_n$
for all $F\in B$.

For every $g,h\in D$ we define a subset $\Omega_{g,h}\subset V$ as
follows:
\[
\Omega_{g,h}=\{ F\in V: \left(\phi_F(g)\right)^{-1}\phi_F(h)\not\in Z\}
\]
The condition
\[
\left(\phi_F(g)\right)^{-1}\phi_F(h) \not\in Z
\]
is equivalent to
\[
g^{-1}h\not\in F(g)ZF(h)
\]
Since $Z$ is finite and the functions $F$ separate the points on $S/U$,
it is clear that $\Omega_{g,h}$ is a dense open subset of $V$.
As a Fr\'echet space $V$ has the Baire property.
It follows that
\[
B\cap \left( \cap_{g\in D}\cap_{h\in D}\Omega_{g,h} \right)
\]
is again dense and in particular not empty.
This implies the statement.
\end{proof}

\begin{proposition}\label{gen-proj}
Let $S$ be a connected complex semisimple Lie group
with center $Z$ and  an exhaustion function $\rho:S\to\R^+$.
Let $K$ be a maximal compact subgroup with Haar measure $\mu$.
Let $T$ be a ``maximal torus'' of $S$ in the sense of algebraic groups,
i.e., a maximal connected reductive commutative complex Lie subgroup
of $S$.

For every discrete subset $D\subset S$ let $\Omega_D$ denote
the set of all $k\in K$ for which the following holds:
\begin{enumerate}
\item
The natural projection map $\pi_k$
from $S$ to $S/kTk^{-1}$ restricted to $D$ is proper, i.e.,
has discrete image and finite fibers.
\item
For every $x,y\in D$ we have $\pi_k(x)\ne\pi_k(y)$ unless
$x^{-1}y\in Z$.
\end{enumerate}

Then $\mu(K\setminus\Omega_D)=0$ for every tame discrete set $D$.

Moreover, there exists a sequence $R_n$ such that every discrete
subset $D\subset S$ is tame if $\#\{p\in D:\rho(p)\le R_n\}\le n$,
i.e., $R_n$ is a ``threshold sequence'' as defined in
definition~\ref{def-ts}.
\end{proposition}

\begin{proof}
We use proposition~\ref{embed-S} to obtain an embedding 
$\alpha:S/T\hookrightarrow W^*$, a proper map $\beta:S\to V^*$
and a linear map $L:V^*\to W^*$. Conjugation by $K$ on $S$ induces
a natural $K$-action on $V^*$, because the dual space $V$ of $V^*$
consists of functions on $S$.

Due to property $(iv)$ of prop.~\ref{embed-S} we may invoke
proposition~\ref{prop-LA} to obtain the following statement:

{\em For every $\delta,r>0$ there exists a number $R>0$ such that
\[
||v||>R \quad\implies\quad
\mu(\{k\in K: ||L(k(v))||< r\} < \delta \ \forall v\in V^*
\]
}
Next we want to apply proposition~\ref{prop-Rn} with $Y=W^*$.

Since $\beta$ is proper, $\rho'(p)=||\beta(p)||$ defines an
exhaustion function on $S$. We observe that for any $R>0$ there
exists a number $R'>0$ such that $\rho(p)>R'$ implies $\rho'(p)>R$.

Therefore proposition~\ref{prop-Rn} combined with its corollary
implies that we have a conull set $\Omega_D'$ for which $(i)$
holds.

We still have to discuss property $(ii)$ of $\Omega_D$.
Let $p,q\in D$ and consider 
\[
C_{p,q}=\{k\in K: \pi_k(p)=\pi_k(q)\}
\]
If $C_{p,q}=K$, then 
\[
pq^{-1}=\cap_k kTk^{-1} =Z
\]
Thus $C_{p,q}$ is a nowhere dense real-analytic subset of $K$
if $pq^{-1}\not\in Z$.
This proves the statement on $\Omega_D$, because nowhere dense
real analytic subsets are of Haar measure zero, and there are only
countably many choices for $(p,q)\in D\times D$.

Finally the statement on the threshold sequence is a direct consequence.
\end{proof}

\begin{proposition}\label{tame-S-proj}
Let $S$ be a semisimple complex Lie group with maximal torus $T$
and a complex-analytic subset $E\subset S$.
Let $D\subset S$ be a tame discrete subset. Then there exists an
automorphism $\phi$  of $S$ such that for both quotients of $S$ by $T$,
the quotient by the right action as well as the quotient by the left
action, the projection map restricts to an injective map from $S$
to a discrete subset of the respective quotient manifold.
In addition, $\phi$ may be chosen such that $\phi(D)\cap E=\{\}$.
\end{proposition}

\begin{proof}
Let $R_n$ be a threshold sequence and let $Z$ denote the center of $S$. 
Due to lemma~\ref{away-center} we may without loss
of generality assume that the natural projection $\tau: S\to S/Z$ is injective on $D$.
Now let $\Omega_D$ be defined as in proposition~\ref{gen-proj}.

Define $\alpha:S\to S$ as $\alpha(g)=g^{-1}$ and $D'=\alpha(D)$.

Due to proposition~\ref{gen-proj} we know that $K\setminus\Omega_D$
and $K\setminus\Omega_{D'})$ are
sets of Haar measure zero.

Now define
\[
U=(K\times K)\setminus\cup_{g\in D}\{ (k,h)\in K\times K: kgh\in E\}
\]
For each $dg\in D$ the set $\{ (k,h)\in K\times K: kgh\in E\}$ is a nowhere
dense real-analytic subset and therefore a set of Haar measure zero.
Since $D$ is countable, it follows that $(K\times K)\setminus U$ is a set of Haar measure
zero.
We observe that
\[
M=\left(\Omega_D\times\Omega_{D'}\right)\cap U
\]
is a conull set and therefore non-empty.
Now we choose $(k,h)\in M$  and define $\phi:g\mapsto kgh^{-1}$.
\end{proof}

\begin{proposition}\label{ex-non-tame}
Every complex semisimple Lie group $S$ admits a non-tame discrete subset.
\end{proposition}

\begin{proof}
$S$  is Stein as a complex manifold.
Due to \cite{JW-LDS} there exists a discrete subset $D_0\subset S$ such
that $\Auto(S\setminus D_0)=\{id\}$.
On the other hand, if $D$ is a tame discrete subset in $S$, then there
exists a biholomorphic self-map $\phi$ of $S$ such that $\pi(\phi(D))$
is discrete in $S/T$ where $T$ is a maximal torus
and $\pi:S\to S/T$ denotes the natural projection. But this implies that 
$S$ admits many holomorphic automorphisms fixing each element of $D$:
We may just take any holomorphic map $f:S/T\to T$ such that
$f(q)=e_T$ for all $q\in\pi(\phi(D))$ and define
\[
x\mapsto \phi^{-1}\left(\phi(x)\cdot f(\pi(\phi(x)))\right).
\]
Thus $\Auto(S\setminus D_0)=\{id\}$ prevents $D_0$ from being tame.
\end{proof}

\section{The case $SL_n\C$}

\begin{definition}
A sequence $A(k)$ in $S=SL_n(\Z)$ is called {\em ``well-placed''}
if the matrix coefficients
$A_{i,j}(k)$ of the elements $A(k)$ fulfill
the following two conditions:
\begin{enumerate}
\item
$A_{i,j}(k)\ne 0$ for all $k\in\N$, $1\le i,j\le n$.
\item
For every $2\le j\le n$ and every $1\le h\le n$ the sequences 
$\alpha_k=|A_{1,h}(k)/A_{j,h}(k)|$
and $\beta_k=|A_{h,1}(k)/A_{h,j}(k)|$
are unbounded and strictly increasing.
\end{enumerate}
\end{definition}

\begin{proposition}\label{well-placed-q}
If a sequence $A(k)$ is ``well-placed'', then
$D=\{A(k):k\in\N\}$ is a tame discrete subset of $S$
and both natural projections
$\pi:S\to S/T$ and $\pi':S\to T\backslash S$ map $D$ injectively
onto a discrete subset of the respective quotient manifold.
\end{proposition}

\begin{proof}
$T$ is the subgroup of diagonal matrices. Its action by
left multiplication on $S$ may be identified
with $(\C^*)^{n-1}$ acting on the coefficients $A_{i,j}$ of elements
$A$ of $SL_n(\C)$  via $A_{j,k}\mapsto A_{j,k}\lambda_{k-1}$ for $k\ge 2$
and
$A_{j,1}\mapsto A_{j,1}\Pi_k \lambda_k^{-1}$.
We may identify $S/T$ with $\left( \P_{n-1}(\C)\right)^n\setminus Z$ where
the projection map $\pi:S\to S/T$ is realized by projecting the columns
of the matrix $A$ to their respective equivalence classes in $\P_n(\C)$;
and where the ``bad locus'' $Z$ consists of those elements
$([v_1],\ldots,[v_n])\in \left( \P_{n-1}(\C)\right)^n$ for which
$v_1,\ldots,v_n$ are not linearly independent.
The assumption of §$A(k)$ being ``well-placed'' implies that
\[
\lim_{k\to\infty}\pi(A(k))=([e],\ldots,[e])
\]
with $e=(1,0\ldots,0)$. Therefore  $\lim_{k\to\infty}\pi(A(k))\in Z$,
i.e., $\{\pi(A(k)):k\in\N\}$ is discrete in $S/T$. Injectivity
follows from the requirement that any sequence 
$k\mapsto |A_{j,1}(k)/A_{j,h}(k)|$ is strictly increasing.
Tameness is due to proposition \ref{pi-tame-is-tame}, taking into account
that $S/T$ is Stein due to the theorem of Matsushima.
The statement for the left quotient are derived in the same way.
\end{proof}

\begin{lemma}\label{lambda}
Let $A(k)$ be  a well-placed sequence $A(k)$.
Assume that $\lambda_i(k)\in\C^*$ are given for $1\le i\le n$, $k\in\N$
such that
\begin{enumerate}
\item
$|\lambda_1(k)|\ge|\lambda_j(k)|$ for all $j\in\{2,\ldots, n\}$, $k\in\N$.
\item
$|\lambda_1(k+1)\lambda_j(k)|\ge |\lambda_1(k)\lambda_j(k+1)|$
for all $j\in\{2,\ldots, n\}$, $k\in\N$.
\item
$\Pi_{i=1}^n \lambda_i(k)=1$ for all $k\in\N$.
\end{enumerate}

Define $B_{i,j}(k)=\lambda_iA_{i,j}(k)$.
Then $B(k)$ is well-placed and equivalent to $A(k)$.
\end{lemma}

\begin{proof}
The conditions on the §$\lambda_i(k)$ ensure that $B(k)$ is
well-placed. 
Proposition~\ref{well-placed-q} implies that the
sequence $A(k)$ is mapped injectively onto a discrete subset of $S/T$
by the natural projection $\pi:S\to S/T$. Now the $\lambda_i(k)$ define
a map $g$ from $D=\{A(k):k\in\N\}$ to $T$ such that $B(k)=A(k)\cdot g(A(k))$
for all $k\in\N$.
Hence $B(k)$ is equivalent to $A(k)$.
\end{proof}

\begin{proposition}\label{tame-implies-well-placed}
Let $D(k)$ be a tame sequence in $S=SL_n(\C)$.
Then there is a sequence $A(k)$ in $S$
such that 
\begin{enumerate}
\item
$D(k)$ and $A(k)$ are equivalent
\item
$A(k)$ is well-placed.
\end{enumerate}
\end{proposition}

\begin{proof}
Due to proposition~\ref{tame-S-proj} we may assume that the projection
map $\pi:S\to S/T$ restricts to an injective map from $D$ onto a
discrete subset of $S/T$. Moreover, we may assume that 
all the matrix coefficients $A_{i,j}(k)$ of every element $A(k)$
are non-zero, again by  proposition~\ref{tame-S-proj}.
(noting that the set $E$ of all $A\in SL_n(\C)$ with at least one
matrix coefficient being zero is a complex analytic subset of $S$.)
$T$ is non-compact and $S/T$ is Stein. Thus for every 
sequence $\zeta(k)\in T$ we can find an automorphism $\phi$ of the complex
manifold $S$ such that $\phi(A(k))=A(k)\cdot\zeta(k)$.
From this we deduce the assertion.
\end{proof}

\begin{proposition}\label{col-equal}
Let $A(k)$, $B(k)$ be well-placed sequences in $S=SL_n(\C)$.
Let $\tau:SL_n(\C)\to\C^n\setminus\{(0,\ldots,0)\}$ be the map
which associated to each matrix its first column vector.

Then there exists a well-placed sequences $C(k)$, $D(k)$ in $S$ such
that 
\begin{enumerate}
\item
$A$ and $C$ are equivalent,
\item
$B$ and $D$ are equivalent,
\item
$\tau(C(k))=\tau(D(k))$ for all $k\in\N$.
\item
$\{\tau(C(k)):k\in\N\}$ is discrete in $\C^n$.
\end{enumerate}
\end{proposition}

\begin{proof}
We want to use lemma~\ref{lambda}.
For this purpose we define recursively sequences $\lambda_j(k)$
and $\mu_j(k)$ in $(\C^*)^n$ such that
\begin{enumerate}
\item
$|\lambda_j(k)|\le 1$ and $|\mu_j(k)|\le 1$ for all $2\le j\le n$.
\item
$|\lambda_j(k)|\le |\lambda_j(k-1)/\lambda_1(k-1)|$ for $2\le j\le n$.
\item
$|\mu_j(k)|\le |\mu_j(k-1)/\mu_1(k-1)|$ for $2\le j\le n$.
\item
$\lambda_j(k)A_{j,1}(k)=\mu_j(k)B_{j,1}$ for $2\le j\le n$.
\item
$\Pi_{j=1}^n\lambda_j(k)=1=\Pi_{j=1}^n\mu_j(k)$
\end{enumerate}
Next we choose
recursively sequences $\tilde\lambda_j(k)$
and $\tilde\mu_j(k)$ in $(\C^*)^n$ such that
\begin{enumerate}
\item
$|\tilde\lambda_j(k)|\le 1$ and $|\tilde\mu_j(k)|\le 1$ for all $2\le j\le n$.
\item
$|\tilde\lambda_j(k)|\le |\tilde\lambda_j(k-1)/\tilde\lambda_1(k-1)|$ for $2\le j\le n$.
\item
$|\tilde\mu_j(k)|\le |\tilde\mu_j(k-1)/\tilde\mu_1(k-1)|$ for $2\le j\le n$.
\item
$\tilde\lambda_1(k)\lambda_1(k)A_{1,1}(k)
=\tilde\mu_1(k)\mu_1(k)B_{1,1}$.
\item
$\Pi_{j=1}^n\tilde\lambda_j(k)=1=\Pi_{j=1}^n\tilde\mu_j(k)$
\end{enumerate}

Now we define sequences $C'(k)$, $C(k)$ as follows:
\begin{enumerate}
\item
$C'_{j,h}(k)=\lambda_j(k)A_{h,j}(k)$
\item
$C_{j,h}(k)=\tilde\lambda_h(k)C'_{j,h}(k)=\tilde\lambda_h(k)\lambda_j(k)A_{h,j}(k)$
\end{enumerate}

Lemma~\ref{lambda} now first implies that $A(k)$ and $C'(k)$ are equivalent
well-placed sequences and then it implies that
$C(k)$ and $C'(k)$ are equivalent well-placed sequences.

Thus $A(k)$ and $C(k)$ are equivalent.

Similarily one verifies that for 
\[
D_{j,h}(k)=\tilde\lambda_h(k)\lambda_j(k)B_{h,j}(k)
\]
the sequences $B(k)$ and $D(k)$ are equivalent.
\end{proof}

\begin{proposition}\label{sln-equiv}
Any two tame sequences in $S=SL_n(\C)$ are equivalent.
\end{proposition}

\begin{proof}
Let $A(k)$, $B(k)$ be two tame sequences.
Each tame sequence is equivalent to a well-placed sequence
(proposition \ref{tame-implies-well-placed}).
Thus we may assume that both $A(k)$ and $B(k)$ are well-placed.
Due to proposition~\ref{col-equal} we may furthermore assume
that for every $k$ the first columns of the matrices $A(k)$ and
$B(k)$ coincide. The map projecting each matrix $A$ to its
first column can be described as $\tau: S\to S/Q\simeq\C^n\setminus\{(0,
\ldots,0)\}$. By proposition~\ref{col-equal} we know that
$\tau(A(k))=\tau(B(k))$ constitutes a sequence in $S/Q$ which is
discrete in $\C^n$. For each $k\in\N$ let $g(k)\in Q$ denote the element
such that $A(k)=B(k)\cdot g(k)$.
Now let $F:\C^n\to Q$ be a holomorphic map such
that $F(\tau(A(k)))=g(k)$ for all $k\in\N$.
Then $F$ defines an automorphism $\phi$ of the manifold $S$ given
by $x\mapsto x\cdot F(\tau(x))$. Via this automorphism $\phi$ the
two sequences $A(k)$ and $B(k)$ are equivalent.
\end{proof}

\begin{corollary}
Let $D$ be a tame discrete subset of $S=SL_n(\C)$.
Then every permutation of $D$ extends to an automorphism
of the complex manifold $S$.
\end{corollary}

\begin{proposition}
Let $D$ be  a discrete subset of $S=SL_n(\C)$.
Let $\pi:S\to \C^n$ denote the linear projection of 
matrices in $S\subset Mat(n\times n,\C)$ onto its first column.

Then $D$ is tame, if $\pi|_D$ has finite fibers and its image is
a discrete subset of $\C^n$.
\end{proposition}

\begin{corollary}\label{union-tame}
Every discrete subset $D$ of $S=SL_n(\C)$ can be realized as
the union of $n$ tame discrete subsets.
\end{corollary}

\begin{proof}\label{column-tame}
Let $\pi_k$ denote the projection onto the $k$-th column and
define
\[
D_k=\{x\in D: ||\pi_k(x)||\ge ||\pi_j(x)||\forall j \}
\]
By construction we have
\[
||\pi_k(x)||\ge \frac{1}{n}||x||\ \ \forall x\in D_k\ \forall k
\]
Since $D$ is discrete, $\{x\in D:||x||<R\}$ is finite for all $R>0$.
Hence $\{x\in D_k:||\pi_k(x)||<\frac 1nR\}$ is finite for all $R>0$. 
It follows
that
$\pi_k(D_k)$ is discrete and that $\pi_k|_{D_k}$ has finite fibers.
Hence each $D_k$ is tame due to proposition~\ref{pi-tame-is-tame}.
\end{proof}

\begin{corollary}\label{torus-tame}
Let $T$ be a torus (i.e.~commutative reductive complex Lie subgroup)
of $SL_n(\C)$  and let $D$ be a discrete subset of $T$.

Then $D$ is a tame discrete subset in $X=SL_n(\C)$.
\end{corollary}

\begin{proof}
Every torus is contained in a maximal torus and all the maximal tori
are conjugate. Hence we may assume that $T$ is the group
of diagonal $n\times n$-matrices with determinant $1$.

Let $A$ denote the matrix given as
\[
A=\begin{pmatrix} 1 &  &  \\ \vdots & \ddots &  \\ 1 & \cdots & 1 
\end{pmatrix}
\]
Then
for each diagonal matrix $M$ with coefficients $\lambda_1,\ldots,\lambda_n)$
we have
\[
\pi( M\cdot A)=(\lambda_1,\lambda_2,\ldots,\lambda_n)
\in Z=\{v\in\C^n: \Pi_iv_i=1\}.
\]
Thus $\psi:M\to M\cdot A$ defines a biholomorphic self-map
of $X$ which maps $T$ biholomorphically onto the closed complex
analytic subset $Z$ of $\C^n$.
Hence $\psi(D)$ (and therefore also $D$) is tame
 due to the proposition.
\end{proof}

\section{Special results for $SL_2(\C)$}
\subsection{$SL_2(\Z[i])$ is tame}

We will see that $SL_2(\Z)$, $SL_2(\Z[i])$ and more generally
$SL_2(\O_K)$ for every imaginary quadratic number field $K$ are
tame discrete subsets in $SL_2(\C)$.

We will need some kind of ``overshears'' on $SL_2(\C)$. 

\begin{lemma}\label{sl-overshear}
For every holomorphic function $\lambda:\C^2\to\C^*$ with
$\lambda(0,w)=1$ for every $w\in\C$ there is a biholomorphic
automorphism $\phi_\lambda$ of the complex  manifold $SL_2(\C)$
such that
\[
\phi_\lambda: \begin{pmatrix} a & c \\ b & d \\
\end{pmatrix}
\mapsto 
\begin{pmatrix} a & c \lambda(a,b)\\ b & d' \\
\end{pmatrix}
\]
with 
\[
d'=\frac{1+bc\lambda(a,b)} {a}
= d + \frac{1-\lambda(a,b)}{a}
\]
if $a\ne 0$.

If we choose  a global coordinate for a fiber $F_v=\pi_1^{-1}(v)\simeq\C$
the restriction of $\phi_\lambda$ to $F_v$ assumes the form
$t\mapsto \lambda(v)t + c$ (with $c$ depending on the choice of the
global coordinate).
\end{lemma}
\begin{proof}
Explicit calculation combined with the observation that
\[
\frac{1-\lambda(a,b)}{a}
\]
has a removable singularity along $a=0$ because we required
$\lambda$ to fulfill the identity $\lambda(0,w)=1\ \forall w\in\C$.
\end{proof}

\begin{proposition}
Let $\pi_i:SL_2(\C)\to\C^2$ denote the projection onto the $i$-th
column.
Let $\Gamma$ be a discrete subset of $SL_2(\C)$ such that
$\pi_1(\gamma)$ is a discrete subset of $\C^2$.

Then $\Gamma$ is tame.
\end{proposition}

\begin{proof}
Define $H=\{(z_1,z_2)\in\C^2: z_1z_2=0\}$. Observe that $\pi_1(AB)=
A\cdot\pi_1(B)$ for $A,B\in SL_2(\C)$. For each $\gamma\in\Gamma$ the
set $U_\gamma=\{A:A\pi_1(\gamma)\not\in H\}$ is Zariski open, because
$\pi_1(\Gamma)\subset\C^2\setminus\{(0,0)\}$. Since $\Gamma$ is
countable, $\cap_{\gamma\in\Gamma}U_\gamma$ is not empty.
Thus by replacing $\Gamma$ with $\{A\gamma:\gamma\in\Gamma$ for a suitably
chosen $A$ we may assume that $\Lambda=\pi_1(\Gamma)\cap H$ is empty.
We fix a bijection $\alpha:\N\to\Lambda$. 

The fibers of $\pi_1$ are the orbits of the principal right action
given as
\[
R_t:\begin{pmatrix} a & c \\ b & d \\
\end{pmatrix}
\ \mapsto\ 
\begin{pmatrix} a & c \\ b & d \\
\end{pmatrix}\cdot
\begin{pmatrix} 1 & t \\ 0 & 1 \\
\end{pmatrix}
=
\begin{pmatrix} a & c+at \\ b & d +bt\\
\end{pmatrix}
\]
Thus there is a natural distance function on each $\pi_1$-fiber
given as $d(A,B)=|t|$ if $B=R_t(A)$.
Using this distance function, choose numbers $\rho_k$ such that
$\pi_1^{-1}(\alpha(k))\setminus\Gamma$ contains a $\rho_k$-ball.
Next we  choose a map $\lambda_0:H\cup\Lambda\to\C$ such that
$\lambda_0|_H\equiv 1$ and
\[
|\lambda(\alpha(k))|\rho_k||\alpha(k)|| > k \forall k\in\N
\]
Using lemma~\ref{sl-overshear} we may from now on assume that
$\rho_k||\alpha(k)||>k$. Observe that for any $v$, $A,B\in\pi_1^{-1}(v)$
we have $||\pi_2(A)-\pi_2(B)||=||v||d(A,B)$.

From this we may deduce that there exists a sequence $t_k$ such that
\[
\pi_2\left( R_{t_k}(\Gamma\cap\pi_1^{-1}(\alpha_k))\right)
\cap \{w\in\C^2: ||w||< k\} =\{\}.
\]

Moreover, chosing $t_k$ sufficiently generic, we may assume that
after applying an automorphism $\psi$ extending the $R_{t_k}$
the projection map $\pi_2:\Gamma\to\C^2$ becomes injective.

Then the assertion follows from proposition~\ref{column-tame}.
\end{proof}

\begin{corollary}\label{sl2zi}
For $K=\Q$ or an imaginary quadratic number field $K$ let
$\Gamma=SL_2(\O_K)$ where $\O_K$ denotes the ring of algebraic
integers in $K$. 

Then $\Gamma$ is tame.
\end{corollary}
\begin{proof}
If $\pi_1$ denotes the projection onto the first column, then
$\pi_1(\Gamma)\subset\O_K\times\O_K\subset\C^2$ and
$\O_K$ is discrete in $\C$.
\end{proof}

In particular, $SL_2(\Z[i])$ is tame.

\section{Miscellenea}

\begin{proposition}
Let $G=SL_n(\C)$
with a tame infinite discrete subset $D$.

Then the automorphism group $\Auto(G\setminus D)$
has uncountably many connected components
(with respect to the compact-open topology).
\end{proposition}

\begin{proof}
Since $D$  is of codimension at least two in $G$, every holomorphic
function on $G\setminus D$ extends to $G$. Now $G$ is Stein and therefore may
be realized as a closed complex submanifold in some $\C^N$.
For every holomorphic automorphism $\phi$ of $G\setminus D$ both $\phi$ and
its inverse $\phi^{-1}$ extend as holomorphic maps $F$ resp.~$F_1$
with values in $\C^N$
through $D$. Since $G\setminus D$ is dense in $G$, we have $F(G)\subset G$
and $F_1(G)\subset G$. Now $F\circ F_1$ and $F_1\circ F$ equal the identity
map on $G\setminus D$ and therefore also on $G$. Hence $F:G\to G$ is
a biholomorphic self-map of $G$ extending $\phi\in\Auto(G\setminus D)$.
Thus every holomorphic automorphism of $G\setminus D$ induces a 
permutation of $D$. Conversely every permutation of $D$ extends to
a holomorphic automorphism. Let $p,q\in D$.
We choose a compact subset $K$ of $G\setminus D$ 
and a Stein open subset $\Omega\subset G$
such that 
\[
p\in \hat K\subset\Omega
\]
(where $\hat K$ denotes the holomorphic convex hull of $K$ in $G$)
and such that $\Omega\cap D=\{p\}$.
Let $\phi_0$ be an automorphism of $G\setminus D$ resp.~its extension
to $G$ with $\phi_0(p)=q$.
Now
\[
W=\{\phi\in \Auto(G\setminus D): \phi(K)\subset\phi_0(\Omega)\}
\]
is open with respect to the compact-open topology on $\Auto(G\setminus D)$.
For every $\phi\in W$ the extension to $G$ (by abuse of notation again denoted
by $\phi$) has the property $\phi(\hat K)\subset\phi_0(\Omega)$, because
$\phi_0(\Omega)$ is Stein and $ \phi(K)\subset\phi_0(\Omega)$.
Since 
\[
\phi(p)\in D\cap\phi_0(\Omega)=\phi_0(\Omega\cap D)=\{\phi_0(p)\},
\]
we obtain that $\phi(p)=q$ for all $\phi\in W$. It follows that
\[
\{\phi\in \Auto(G\setminus D):\phi(p)=q\}
\]
is open in $\Auto(G\setminus D)$.
This in turn implies that we have a continuous surjective map
from $\Auto(G\setminus D)$ to $Perm(D)$ where $Perm(D)$ is endowed
with a totally disconnected topology, namely the topology which has
\[
W_{p,q}=\{\phi:\phi(p)=q\} 
\]
as a basis of topology.
Because $Perm(D)$ is uncountable for an infinite countable set $D$
the assertion follows.
\end{proof}

It is easily verified that a discrete subset $D$ in $\C^n$ is tame
if and only if there is a biholomorphic self-map $\psi$ of $\C^n$
such that $\psi(D)$ is contained in a complex line.
The complex lines through the origin in $\C^n$ are precisely
the one-dimensional connected  complex Lie subgroups of the
additive group $(\C^n,+)$.
Below we present a statement in the same spirit for semisimple
Lie groups.

\begin{proposition}\label{one-para}
An infinite discrete subset of $S=SL_n(\C)$ is tame if and only
if there exists an connected one-dimensional complex
algebraic subgroup $A$ and an
automorphism of the complex manifold $S$ such that $\phi(D)\subset A$.
\end{proposition}
\begin{proof}
Since $A$ is unbounded, it contains a tame discrete subset $D'$
(proposition~\ref{gen-proj} and proposition~\ref{unbounded}).
This proves one direction, because any two tame discrete subsets
of $S$ are equivalent (proposition~\ref{sln-equiv}).

For the opposite direction, let $D\subset A$ be a discrete subset.

If $A$ is reductive, i.e., if $A\simeq\C^*$, then $D$ is tame
due to corollary~\ref{torus-tame}.

It remains to discuss the case where $A\simeq\C$.
Let $p_k:S\to\C^n$ denote the map which associates to each matrix
its $k$-th column. Choose a $k$ such that $p_k$ is not constant on $A$.
Then $p_k:A\simeq\C\to\C^n$ is a non-constant algebraic morphism.
Note that every algebraic morphism from $\C$ to $\C^n$ is given by
polynomials and therefore is a proper map. 
Hence $D$ is tame due to proposition~\ref{column-tame}.
\end{proof}

\section{Some open questions}

\begin{itemize}
\item
Is every discrete subgroup in a Stein complex Lie group tame?
\item
To which extent do our results for $SL_n(\C)$ extend to arbitrary Stein
complex Lie groups, and more generally, to arbitrary homogeneous manifolds,
perhaps presuming ``flexibility'' or ``density property'' ?
\item
Does every non-compact complex manifold admit a non-tame infinite discrete 
subset?
\end{itemize}

Instead of considering the full automorphism group one may also discuss
certain subgroups, e.g., the subgroup preserving some volume form.

\end{document}